\documentclass[final,reqno]{siamltex}
\usepackage{amsmath}
\usepackage[table]{xcolor}
\usepackage{subfig}
\usepackage{graphicx}
\usepackage{amsfonts}
\usepackage{tikz}
\usepackage{placeins}
\usepackage[table]{xcolor}
\usepackage{booktabs}
\usepackage{epsfig,epsf,psfrag}
\usepackage{multirow}
\usepackage[left=1in, right=1in, top=1in, bottom=1in]{geometry}

\newtheorem{remark}{Remark}

\title{A fast direct solver for high frequency scattering from a large cavity in two dimensions}
\author{Jun Lai\thanks{Courant Institute of Mathematical Sciences, New York University, NY 10012} \and Sivaram Ambikasaran\footnotemark[1] \and Leslie F. Greengard\footnotemark[1] \thanks{Simons Center for Data Analysis, Simons Foundation, New York, NY 10010}}
\date{}

\begin{document}
\maketitle
\pagestyle{myheadings}
\markboth{Jun Lai, Sivaram Ambikasaran, Leslie F. Greengard}{Fast direct solver for high frequency scattering from a large cavity in $2$D}
\begin{abstract}
We present a fast direct solver for the simulation of electromagnetic scattering from 
an arbitrarily-shaped, large, empty cavity embedded in an infinite perfectly conducting half space. 
The governing Maxwell equations are reformulated as a well-conditioned second kind integral equation 
and the resulting linear system is solved in nearly linear time using a hierarchical matrix factorization
technique. We illustrate the performance of the scheme with several numerical examples for 
complex cavity shapes over a wide range of frequencies.
\end{abstract}

\begin{keywords}
electromagnetic scattering, fast direct solver, large cavity
\end{keywords}

\section{Introduction}

Electromagnetic scattering from large cavities has been studied extensively over the 
years~\cite{HGA2,HGA1,HGA3,GW1,GWZ,GLin,JIN1,JIN2}, due to the widespread presence of cavities
in practical settings. It is of particular interest in radar cross section (RCS) analysis,
both for accentuating a signal (in tracking a vehicle) and for its mitigation (in 
electromagnetic interference and stealth design). Cavities play an important role in these contexts
because of the well-known fact that the intensity of the echo wave is often
dominated by scattering from cavity-like components~\cite{RCS}, such as the exhaust nozzle
or engine inlet of an aircraft. 
In the context of design, we refer the reader to \cite{BJL2,BJL1},
where RCS enhancement or reduction was carried out through the use of an optimization procedure
based on a Newton-type method. At each iteration, a large-scale scattering problem has to be solved,
dominating the net cost.
A second application is non-destructive testing to determine the shape of an existing cavity. 
The corresponding stability analysis was initially studied in~\cite{GKZ}. 
Numerical inversion again requires an efficient 
solver that works over a range of frequencies. To obtain fine features inside the cavity,
high frequency measurements are required, making the problems large, oscillatory
and progressively more ill-conditioned.
In short, efficient and accurate numerical methods for modeling
in the presence of \emph{complex, arbitrarily-shaped cavities} are becoming essential.

Integral equation methods are 
very natural choices for the solution of scattering problems because they discretize the 
scatterer alone and are able to impose outgoing radiation conditions without the need for
truncating the spatial domain and imposing artificial boundary conditions. 
In the case of cavities in a conducting half-space,
a variety of integral formulations exists and we refer the readers to 
Bao \textit{et al}~\cite{HGA2}, Asvestas \textit{et al}~\cite{AK1994}, 
Willers \textit{et al}~\cite{WW1987}, Chandler-wilde \textit{et al}~\cite{Chandle05} and 
the references therein for a complete discussion. 
Here, we simply note that the
choice of integral equation has a great impact on the accuracy of the numerical 
discretization~\cite{Helsing2008} and the condition number of the resulting linear system. 
In this article, we propose a boundary integral formulation that leads in a straightforward
way to a well-posed, high-order discretization. Like the formulation of~\cite{HGA2}, 
we impose continuity conditions on a ``transparent" dome covering the cavity, which
reduces the problem to one posed in a bounded domain. 
A principal difference is that, in our case, we introduce a non-physical charge density on
part of the ground plane which permits high order accuracy - avoiding the difficulties 
introduced by signularities at \textit{triple-points}, points that lie at the intersection of more 
than two subdomains (see Fig. \ref{figure_cavity_shape}).
 
Historically, the major challenge with integral formulations has been that the corresponding 
linear system is dense so that solving by conventional linear solvers is expensive, requiring 
$\mathcal{O}(N^3)$ work for a $N \times N$ linear system. That complexity barrier was overcome
by using iterative techniques based on Krylov subspace methods~\cite{arnoldi1951principle, hestenes1952methods, saad1986gmres, paige1975solution, van1992bi, freund1991qmr, freund1993transpose},
coupled with fast matrix-vector product techniques such as the fast multipole method 
(FMM)~\cite{coifman1993fast,greengard1997new,greengard1987fast,cheng1999fast}, 
tree based algorithms~\cite{barnes1986hierarchical}, 
panel clustering~\cite{hackbusch1989fast}, 
FFT~\cite{cooley1965algorithm}, 
wavelet based methods~\cite{mallat1989theory}, and a host of others. 
Though these techniques have a number of attractions, the number of iterations required to 
achieve a specified accuracy is highly problem-dependent. In the context of 
electromagnetic scattering, if the geometry is complicated and if the frequency of the incident 
field is high, the number of iterations can be extremely large, so that the methods no longer
behave linearly and require large amounts of storage. Recently, there has been an increasing 
focus on fast direct solvers~\cite{greengard2009fast} for dense linear systems arising from 
integral equations. This is an important and active area of research and we refer the reader 
to three relevant (and related) formualtions - those based on hierarchical off-diagonal 
low-rank matrices (HODLR)~\cite{ambikasaran2013fast, amirhossein2014fast, kong2011adaptive}, 
hierarchically semi-separable (HSS) or hierarchically block-separable (HBS) matrices~\cite{chandrasekaran2006fast1,chandrasekaran2006fast,ho2012fast,martinsson2009fast,martinsson2005fast}, $\mathcal{H}$, 
and $\mathcal{H}^2$ matrices~\cite{bebendorf2005hierarchical,borm2003hierarchical,borm2003introduction,hackbusch2001introduction}. In this article, since the integral equation we are solving is on a 
$1$D manifold, we rely on the fast direct solver discussed in 
Ambikasaran and Darve~\cite{ambikasaran2013fast}, which scales almost 
linearly (as $\mathcal{O}(N \log^2 N)$) in the number of unknowns even for problems that
are hundreds of wavelengths in size.

We restrict our attention here to time-harmonic scattering over a wide range of frequencies 
for a $2$D cavity embedded in a ground plane, as shown in Figure~\ref{figure_cavity_shape}, 
where the boundary of the cavity and the ground plane are perfectly conducting. 
The cavity is \textit{empty}, i.e., the permittivity $\epsilon$ and permeability $\mu$ are 
constant everywhere inside the cavity and equal to that of the upper half-space.

\begin{figure}[!htbp]
\begin{center}
\includegraphics[scale=0.5]{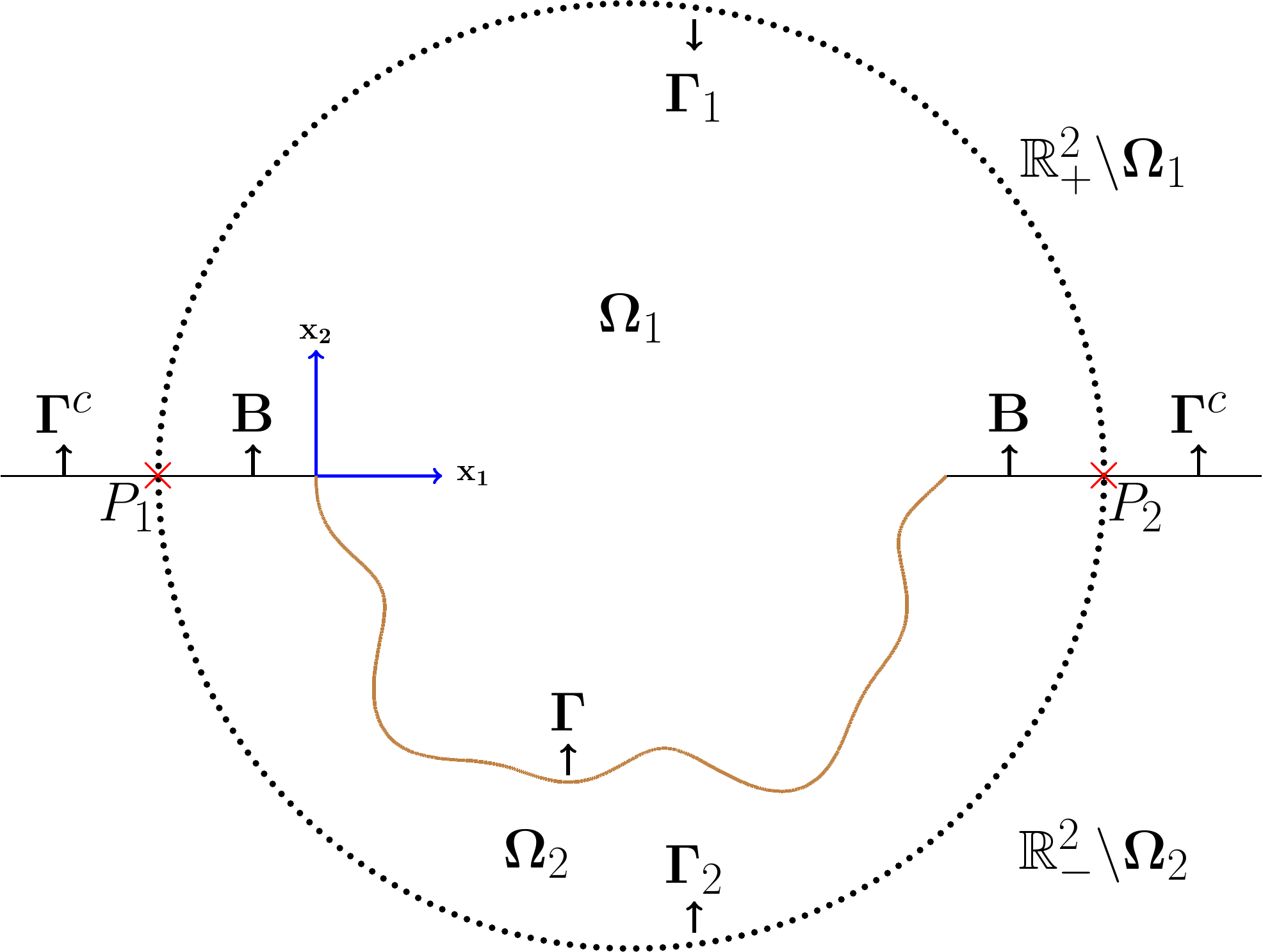}
\caption{Geometry of the cavity shape. Note that $P_1$ and $P_2$ are triple-points.
They lie at the intersection of domains $\Omega_1$, $\Omega_2$, $R_+^2\setminus \Omega_1$, and 
$R_-^2\setminus \Omega_2$.}
\label{figure_cavity_shape}
\end{center}
\end{figure}

It is well known that, in $2$D, the wave can be decomposed into 
a transverse magnetic (TM) polarization component and a transverse electric (TE) polarization component. 
Maxwell's equations reduce in this setting to a scalar Helmholtz 
equation for the $z$-component the magnetic (TE) or electric (TM) field, respectively.
We restrict our attention to the TM polarization and denote the 
$z$-component of the electric field by $u$, but note that the treatment of the TE 
polarization is very similar. 

We seek to 
determine the unknown scattered wave $u^s$, emanating from the cavity $\Omega_1$ in response to 
a known incident wave $u^{i}$. The 
governing equation is the 
Helmholtz equation:
\begin{equation}
\label{holm}
\left\{\begin{array}{rl}
\Delta u+k^2u&= 0 \mbox{ in } \Omega_1\cup R_+^2  \\
u &= 0 \mbox{ on } \Gamma\cup \Gamma^c. 
\end{array}\right.
\end{equation} 
where $R_+^2$ denotes the upper half space, $\Gamma$ is the boundary of the cavity,  $\Gamma^c$ is the ground plane and $k$ is the wavenumber. The wavenumber $k$ depends on the given angular frequency $\omega$ (time harmonic dependence is $e^{-i\omega t}$), i.e., $k= \omega\sqrt{\varepsilon\mu}$. 
where, as noted above, the permittivity $\epsilon$ and permeability $\mu$ are  assumed to be 
constant.

The total field $u$ is considered as the summation of three parts~\cite{JIN3}: the incident field $u^{i}$, the reflected field $u^{r}$ and the scattered field $u^s$. 
$u^i$ is known, and typically defined as a plane wave or the wave induced by known sources 
in the upper half-space.
$u^r$ denotes the wave reflected by a perfectly conducting half-space (without a cavity) and
can be computed analytically from $u^i$.
The only unknown, therefore, is the scattered field $u^s$, which satisfies the 
Helmholtz equation~\eqref{holm} as well as the Sommerfeld radiation condition:
\begin{equation}\label{eq4}
\lim_{r\rightarrow \infty} \sqrt{r} \bigg(\frac{\partial u^s}{\partial r}-iku^s \bigg) = 0 \, ,
\end{equation}
where $r = \sqrt{x_1^2+x_2^2}$. 

The paper is organized as follows. In Section~\ref{section_formulation}, we describe
the cavity problem in more detail. We propose a new second kind integral formulation that 
leads to a stable numerical algorithm. In addition, we prove uniqueness for the equation,
showing that the the formulation is well-posed.
Section~\ref{section_numerical_discretization} discusses the numerical discretization of the 
integral equation, while 
Section~\ref{section_fast_direct_solver} discusses the hierarchical, fast, direct solver
that relies on specific properties of the matrix arising from discretization of our integral equation. 
Numerical examples are presented in Section~\ref{section_numerical_results}, demonstrating
the efficiency of the method and Section~\ref{section_conclusion} contains some concluding
remarks.

\section{The cavity problem}
\label{section_formulation}

Consider a time-harmonic wave incident on the cavity $\Omega_1$, embedded in an infinite ground 
plane $\Gamma^c$, illustrated in Fig. \ref{figure_cavity_shape}. 
The boundary components $\Gamma$ and $\Gamma^c$ are both assumed to be perfect conductors. 
We also assume that $\Gamma$ is defined by a piecewise smooth curve and that none of the 
corners between adjacent smooth components involve cusps (that is,
the angle at each corner is greater than $0$ and less than $2\pi$).
 
When the incoming field is a plane wave, we have 
$u^{i} = e^{ik(\cos\theta x_1-\sin\theta x_2)}$,
and $u^{r} = e^{ik(\cos\theta x_1+\sin\theta x_2)}$,
where $\theta$ is the angle of incidence with respect to the positive $x_1$ axis. 
In this case, the scattered wave $u^s$ satisfies the equation:
\begin{equation}\label{eq5}
\left\{\begin{array}{rl}
\Delta u^s+k^2u^s&= 0 \mbox{ in } \Omega_1\cup R_+^2\\
u^s &= g \mbox{ on } \Gamma\cup \Gamma^c
\end{array}\right.
\end{equation}
along with the Sommerfeld radiation condition \eqref{eq4}, where $g = -(u^{i}+u^{r})$. 

\subsection{The integral formulation}

A variety of integral formulations have been proposed in the literature for the cavity problem. 
Those discussed in~\cite{HGA2,WW1987}, for example, are based on Green's identities. 
Another formulation, based on the method of images, is discussed in~\cite{AK1994}. Our 
formulation is based on potential theory, with the goal of deriving a second kind 
integral equation that will lead to a high order numerical scheme, even in the presence
of sharp corners. In all these approaches, it is convenient to introduce an artificial 
interface close to the cavity, which we denote here by $\Gamma_1$ (illustrated by the 
dotted line in Fig. \ref{figure_cavity_shape}).
One does not need to introduce artificial boundary conditions
on this interface, however. One simply imposes continuity conditions so that 
there is a representation for the scattered wave in the exterior of $\Gamma_1$ as an integral
over the artificial interface itself.
$\Gamma_1$ needs to be chosen so that its reflection with respect to the ground plane, $\Gamma_2$, 
does not intersect the cavity. We denote by $B$ the part of the ground plane $\Gamma^c$ that 
connects $\Gamma_1$ and $\Gamma$ and define $\Omega_1$ to be the domain enclosed by 
$\Gamma\cup B\cup \Gamma_1$. Its complement is the unbounded domain $R_+^2\backslash \Omega_1$.

We let $\Phi(\mathbf{x},\mathbf{y})$ denote the (radiating) free space Green's function for the 
Helmholtz equation and we let $\Phi^{H}(\mathbf{x},\mathbf{y})$ 
denote the Green's function satisfying homogeneous Dirichlet conditions on the
half space. It is well-known that
\begin{align}
\Phi(\mathbf{x},\mathbf{y}) &= \frac{i}{4}H^{(1)}_0(k|\mathbf{x}-\mathbf{y}|) \, , \\
\Phi^{H}(\mathbf{x},\mathbf{y}) &= \frac{i}{4}H^{(1)}_0(k|\mathbf{x}-\mathbf{y}|) - \frac{i}{4}H^{(1)}_0(k|\mathbf{x}-\mathbf{y}'|) \, ,
\end{align}
where 
$H^{(1)}_0$ is the Hankel function of the first kind of order zero and 
$\mathbf{y}'=(y_1,-y_2)$ is the image point for $\mathbf{y} = (y_1,y_2)$.

A natural representation for $u^s$ in the unbounded domain $R_+^2\backslash \Omega_1$ is
\begin{equation}\label{eq6}
u^s = \mathcal{S}^{H}_{\Gamma_1}\sigma + \mathcal{D}^{H}_{\Gamma_1} \mu,  \mbox{ for } \mathbf{x}\in R_+^2\backslash \Omega_1 \, ,
\end{equation}
where $\mathcal{S}^{H}$ and $\mathcal{D}^{H}$ are single and double layer potentials,
using the half-space Green's function:
\begin{align}
\mathcal{S}^{H}_{\Gamma_1}\sigma &= \int_{\Gamma_1}\Phi^H(\mathbf{x},\mathbf{y})\sigma(\mathbf{y}) ds_{\mathbf{y}} \, , \\
\mathcal{D}^{H}_{\Gamma_1}\mu &=   \int_{\Gamma_1}\frac{\partial \Phi^H(\mathbf{x},\mathbf{y})}{\partial n(\mathbf{y})}\mu(\mathbf{y}) ds_{\mathbf{y}}.
\end{align}
Here the unit normal vector $n$ on $\Gamma_1$ is assumed to be oriented toward the exterior of 
$\Omega_1$. $\sigma$ and $\mu$ are, for the moment, unknown density functions on the 
boundary $\Gamma_1$.
\begin{remark}
By the methods of images, we can rewrite $\mathcal{S}^{H}$ as:
\begin{equation}
\mathcal{S}^{H}_{\Gamma_1}\sigma = \mathcal{S}_{\Gamma_1}\sigma 
-\mathcal{S}_{\Gamma_2}\sigma
\end{equation}
where $S_{\Gamma}$ denotes single layer potential based on the 
free-space Green's function
\[ S_{\Gamma}\sigma = \displaystyle \int_{\Gamma} \Phi(\mathbf{x},\mathbf{y}) \sigma(\mathbf{y}) ds_{\mathbf{y}} \, , \]
and for $\mathbf{x}' \in\Gamma_2$, 
$\sigma(\mathbf{x}') =  \sigma(\mathbf{x})$, where $\mathbf{x}'$ is the image point for 
$\mathbf{x}$. Clearly, a similar expression holds for $\mathcal{D}^{H}$ using the free-space
double layer potential
\[ D_{\Gamma}\mu = \displaystyle \int_{\Gamma} \dfrac{\partial \Phi(\mathbf{x},\mathbf{y})}{\partial n(\mathbf{y})} \mu(\mathbf{y}) ds_{\mathbf{y}} \, . \] 
\end{remark}

Suppose now that we represent the scattered field for $\mathbf{x}\in \Omega_1$ by
\begin{equation}\label{eq7}
u^s(\mathbf{x}) = \mathcal{S}_{\Gamma_1}\sigma+ \mathcal{D}_{\Gamma_1} \mu+\mathcal{D}_{B}\mu + 
\mathcal{D}_{\Gamma}\mu  \, .
\end{equation}

Imposing the continuity of the potential $u$ and its normal derivative across
$\Gamma_1$, and using the
representations \eqref{eq6} and \eqref{eq7}, leads to 
\begin{equation}\label{eq8}
\left\{\begin{array}{c}
\mu -\mathcal{S}_{\Gamma_2}\sigma-\mathcal{D}_{\Gamma_2}\mu-\mathcal{D}_{B}\mu - \mathcal{D}_{\Gamma}\mu = 0,\\
\sigma +\mathcal{N}_{\Gamma_2}\sigma+\mathcal{T}_{\Gamma_2}\mu+\mathcal{T}_{B}\mu + \mathcal{T}_{\Gamma}\mu = 0, 
\end{array} \right.
\mbox{ for } \mathbf{x}\in \Gamma_1 \, .
\end{equation}
For this, one needs to use standard jump relations for 
the layer potentials $\mathcal{S}$ and $\mathcal{D}$ \cite{Cot2,Cot}. 
The operators $\mathcal{N}$ and $\mathcal{T}$ are defined as the normal derivatives of 
$\mathcal{S}$ and $\mathcal{D}$, respectively:
\begin{align}
\mathcal{N}_{\Gamma}\sigma &= \int_{\Gamma}\frac{\partial \Phi(\mathbf{x},\mathbf{y})}{\partial n(\mathbf{x})}\sigma(\mathbf{y})ds_{\mathbf{y}} \, , \\
\mathcal{T}_{\Gamma}\mu &= \int_{\Gamma}\frac{\partial^2\Phi(\mathbf{x},\mathbf{y})}{\partial n(\mathbf{x})\partial n(\mathbf{y})}\mu(\mathbf{y})ds_{\mathbf{y}} \, .
\end{align}
Note that the operator $\mathcal{T}$ is \textit{hypersingular} on $\Gamma$, with its value 
interpreted in the Hadamard finite part sense.
Similarly, using the representation~\eqref{eq7}, and letting $x$ approach the boundary $\Gamma\cup B$, 
yields
\begin{equation}\label{eq9}
-\frac{1}{2}\mu+\mathcal{S}_{\Gamma_1}\sigma+ \mathcal{D}_{\Gamma_1} \mu+\mathcal{D}_{B}\mu + \mathcal{D}_{\Gamma}\mu=g, \mbox{ for } \mathbf{x}\in \Gamma\cup B \, .
\end{equation}

Combining eqs.~\eqref{eq8} and~\eqref{eq9} we obtain a closed system for $\sigma$ and $\mu$. 
Unfortunately, careful analysis shows that it is not a 
Fredholm equation of the second kind. In particular, the hypersingular operators are 
unbounded at the
\textit{triple-points} $P_1$ and $P_2$ in Fig. \ref{figure_cavity_shape}. This
leads to a failure of convergence. Following the approach of \cite{GHL2014,GL2014}, we remedy the situation 
by using a non-physical representation  
near triple-points so that the resulting integral equation involves only 
the {\em difference} of two hypersingular kernels, which is easily seen to be compact. 
Thus, we propose the following, new formulation:
\begin{equation}\label{eq10}
\left\{\begin{array}{l}
u^s = \mathcal{S}^{H}_{\Gamma_1}\sigma + \mathcal{D}^{H}_{\Gamma_1} \mu + \mathcal{D}_{B} \mu
\qquad \qquad  \mbox{ for } \mathbf{x}\in R_+^2\backslash \Omega_1 \, , \\
u^s = \mathcal{S}^{H}_{\Gamma_1}\sigma+ \mathcal{D}^{H}_{\Gamma_1} \mu+\mathcal{D}_{B}\mu + \mathcal{D}_{\Gamma}\mu \ \ \ \mbox{ for } \mathbf{x}\in \Omega_1\, . 
\end{array}
\right.
\end{equation}

We refer to the representation as non-physical because the terms we have added
($\mathcal{D}_{B} \mu$ for the exterior field and 
$-\mathcal{S}_{\Gamma_2} \sigma - \mathcal{D}_{\Gamma_2} \mu$ for the interior field) that
involve boundary components that do not actually impinge on the domain.
The advantage we obtain is a cancellation of the hypersingular terms.
To see this, imposing the continuity and boundary conditions as above, and using the standard
jump relations, we obtain the system:
\begin{equation}\label{eq11}
\left\{\begin{array}{rl}
\mu - \mathcal{D}_{\Gamma}\mu &= 0, \quad \mbox{ for } \mathbf{x} \in \Gamma_1\\
\sigma + \mathcal{T}_{\Gamma}\mu &= 0,\quad \mbox{ for } \mathbf{x} \in \Gamma_1\\
-\frac{1}{2}\mu+\mathcal{S}^{H}_{\Gamma_1}\sigma+ \mathcal{D}^{H}_{\Gamma_1} \mu+\mathcal{D}_{B}\mu + \mathcal{D}_{\Gamma}\mu &= g, \quad \mbox{ for } \mathbf{x}\in B\cup \Gamma
\end{array}\right. \, .
\end{equation}

Although the representation~\eqref{eq10} is slightly more involved than \eqref{eq6} and \eqref{eq7}, 
note that the resulting system of integral equations is actually simpler. 
In particular, the first two equations in \eqref{eq11} imply that $\mu$ and $\sigma$ on $\Gamma_1$ 
are directly determined by the value of $\mu$ on $\Gamma$. 
More importantly, the system is a Fredholm equation of the second kind in a suitably-defined space.
The analysis is somewhat technical, since operators that are compact on smooth domains are only 
bounded on domains with corners. That is the case here, and we refer the reader to 
\cite{Brem2012} for further details.
In the next section, we discuss existence and uniqueness, followed by a discussion of numerical
discretization. 

\begin{remark}
In \cite{AK1994}, a related formulation was developed that used only
the unknown $\mu$ on $B\cup \Gamma$ in \eqref{eq10}. This, however, requires that the cavity remain
strictly below the ground plane. Our approach works for an arbitrary cavity shape. 
\end{remark}

\subsection{Existence and uniqueness}
To simplify the analysis, we assume that the boundary of the cavity 
$B\cup\Gamma$, $\Gamma_1$ and $\Gamma_2$ are $C^2$ smooth curves, in which case the solution $u$ 
to eq. \eqref{eq5} lies in $C^2(R_+^2\cup \Omega_1)\cup C(\overline{R_+^2\cup \Omega_1})$ (see 
\cite{WW1987}) so that the standard Fredholm theory applies. Thus,
to establish uniqueness for the system \eqref{eq11}, it suffices to show that 
$\mu=0$ and $\sigma= 0$ if $g=0$.
Extension of the proof to the piecewise-smooth case is straightforward,
and consists largely in a switch to the corresponding Sobolev space in order that the Fredholm
theory be applicable.

\begin{theorem} Given $k>0$, the system
\begin{equation}\label{eq12}
\left\{\begin{array}{rl}
\mu - \mathcal{D}_{\Gamma}\mu &= 0, \quad \mbox{ for } \mathbf{x} \in \Gamma_1\\
\sigma + \mathcal{T}_{\Gamma}\mu &= 0,\quad \mbox{ for } \mathbf{x} \in \Gamma_1\\
-\frac{1}{2}\mu+\mathcal{S}^{H}_{\Gamma_1}\sigma+ \mathcal{D}^{H}_{\Gamma_1} \mu+\mathcal{D}_{B}\mu + \mathcal{D}_{\Gamma}\mu &= 0, \quad \mbox{ for } \mathbf{x}\in B\cup \Gamma
\end{array} \right.
\end{equation}
has only the trivial solution for $\mu \in C(B) \cup C(\Gamma)\cup C(\Gamma_1)$ and $\sigma \in C(\Gamma_1)$.
\end{theorem}
\begin{proof}
Motivated by the analysis of \cite{HGA2},
we let
\begin{align}\label{eq15}
v= \left\{\begin{array}{l}
\mathcal{S}^{H}_{\Gamma_1}\sigma + \mathcal{D}^{H}_{\Gamma_1} \mu + \mathcal{D}_{B} \mu,  \mbox{ for } \mathbf{x}\in R^2\backslash (\Omega_1\cup \Omega_2) \\
\mathcal{S}^{H}_{\Gamma_1}\sigma+ \mathcal{D}^{H}_{\Gamma_1} \mu+\mathcal{D}_{B}\mu + \mathcal{D}_{\Gamma}\mu, \mbox{ for } \mathbf{x}\in \Omega_1\cup \Omega_2
\end{array} \right.
\end{align}
where $\Omega_2$ is the area enclosed by $\Gamma\cup B\cup \Gamma_2$. 

Combining \eqref{eq15} and eq. \eqref{eq12}, we have
\begin{equation}\label{eq13}
\quad v = 0 \mbox{ on } \Gamma^c, \quad \lim_{\substack{x\in \Omega_1 \\ x\rightarrow \Gamma\cup B}} v(x) = 0 \, ,
\end{equation}
and the jump conditions
\begin{align}\label{eq14}
\begin{split}
[v]|_{\Gamma_1} &= 0, \\
\bigg[\frac{\partial v}{\partial n}\bigg]|_{\Gamma_1} &= 0,
\end{split} 
\end{align}
where $[v]|_{\Gamma_1}$ denotes the jump of the function $v$ across the boundary $\Gamma_1$.

We first prove $v$ is identically zero in $R_+^2\cup \Omega_1$. For this, choose a sufficiently large 
half disk $D$ above the ground plane that contains the boundary $\Gamma_1$. For the area $D\backslash \Omega_1$, the boundary consists of $\Gamma_1$, part of the ground plane (still denoted by $\Gamma^c$) and a half circle, denoted by $\partial D$. Applying Green's theorem to $v$ in $D\backslash \Omega_1$ and $\Omega_1$, respectively, we obtain
\begin{equation} \label{green1}
\int_{\Omega_1} v\triangle\overline{v} + \nabla v\nabla\overline{v} dx = \int_{B\cup \Gamma} v\frac{\partial\overline{v}}{\partial n}ds + \int_{\Gamma_1} v\frac{\partial\overline{v}}{\partial n}ds
\end{equation}
and
\begin{equation} \label{green2}
\int_{D\backslash \Omega_1} v\triangle\overline{v} + \nabla v\nabla\overline{v} dx = \int_{\partial D} v\frac{\partial\overline{v}}{\partial n}ds - \int_{\Gamma_1} v\frac{\partial\overline{v}}{\partial n}ds+\int_{\Gamma^c} v\frac{\partial\overline{v}}{\partial n}ds \, .
\end{equation}
Adding eq. \eqref{green1} and eq. \eqref{green2}, together with \eqref{eq13} and \eqref{eq14},
yields
\begin{equation}
\int_{\partial D} v\frac{\partial\overline{v}}{\partial n}ds = \int_{D\backslash \Omega_1} (-k^2v\overline{v}+ \nabla v\nabla\overline{v}) dx + \int_{\Omega_1} (-k^2v\overline{v} + \nabla v\nabla\overline{v} )dx,
\end{equation}
which implies that
\begin{equation}
\text{Im}\bigg(\int_{\partial D} v\frac{\partial\overline{v}}{\partial n}ds\bigg) = 0 \, .
\end{equation}
It follows from Rellich's theorem \cite{Cot2}, applied to the half space, and the unique 
continuation property \cite{Cot}, that $v=0$ in $R_+^2\cup \Omega_1$.

We next show $v(x)$ satisfies the following equation in $\Omega_2$
\begin{align}\label{eq18}
\left\{\begin{array}{rl}
\Delta v + k^2 v &= 0,\mbox{ for } x \in \Omega_2\\
v &=\mathcal{D}_{\Gamma}\mu+\mathcal{D}_{\Gamma'}\mu,\ \mbox{ on } \Gamma_2\\
\frac{\partial v}{\partial n} &= \mathcal{T}_{\Gamma}\mu+\mathcal{T}_{\Gamma'}\mu,\ \mbox{ on } \Gamma_2\\
v &=\mu,\ \mbox{ on } B\cup\Gamma  \\
\frac{\partial v}{\partial n} & = 0, \ \mbox{ on } B\cup \Gamma \, ,
\end{array}\right.
\end{align}
where $\Gamma'$ is the image curve of $\Gamma$ with respect to the ground plane and 
$\mu(\mathbf{x})$ on $\Gamma'$ is the image source of $\mu(\mathbf{x}')$ on $\Gamma$. 

It follows from \eqref{eq15} that we have the following jump condition on $\Gamma_2$:
\begin{align}\label{eq16}
\begin{split}
[v]_{\Gamma_2} = \mu + \mathcal{D}_{\Gamma}\mu\\
\bigg[\frac{\partial v}{\partial n}\bigg]_{\Gamma_2} = -\sigma + \mathcal{T}_{\Gamma}\mu.
\end{split}
\end{align}
Since $v=0$ in $R_+^2\cup \Omega_1$, by symmetry, $v = 0$ in $R^2\backslash (\Omega_1\cup \Omega_2)$. 
From \eqref{eq12}, $\mu(\mathbf{x}) =\mu(\mathbf{x}')= \mathcal{D}_{\Gamma}\mu(\mathbf{x}')$ for $\mathbf{x}\in \Gamma_2$. However, by the definition of the image curve $\Gamma'$ and the source $\mu$ on $\Gamma'$, it is easy to see that $\mathcal{D}_{\Gamma}\mu(\mathbf{x}') =\mathcal{D}_{\Gamma'}\mu(\mathbf{x})$. Similarly, one can see that $\sigma(\mathbf{x}) = -\mathcal{T}_{\Gamma'}\mu(\mathbf{x})$ for $\mathbf{x}\in \Gamma_2$. Combining eqs \eqref{eq12} and \eqref{eq16} leads to
\begin{align}
\begin{split}
\lim_{\substack{\mathbf{x}\in \Omega_2 \\ \mathbf{x}\rightarrow \Gamma_2}} v &= \mathcal{D}_{\Gamma}\mu+\mathcal{D}_{\Gamma'}\mu \, , \\
\lim_{\substack{\mathbf{x}\in \Omega_2 \\  \mathbf{x}\rightarrow \Gamma_2}} \frac{\partial v}{\partial n} &= \mathcal{T}_{\Gamma}\mu+\mathcal{T}_{\Gamma'}\mu\, . 
\end{split}
\end{align}
We therefore obtain the first two boundary conditions for $v$ in \eqref{eq18}. 
The boundary conditions on $B\cup \Gamma$ are obtained from the jump condition
\begin{align}\label{eq17}
\begin{split}
[v]_{B\cup \Gamma} = \mu \, ,\\
\bigg[\frac{\partial v}{\partial n}\bigg]_{B\cup \Gamma} = 0\, ,
\end{split}
\end{align}
and the fact that $v$ is zero in $\Omega_1$.

We now prove that the solution to \eqref{eq18} is identically zero. Let
\begin{equation}
w(x) = \mathcal{D}_{\Gamma}\mu + \mathcal{D}_{\Gamma'}\mu,\quad \forall x \in R_-^2\backslash\Omega_2\, .
\end{equation}
It is easy to see $\frac{\partial w(x)}{\partial n}= 0$ on $\Gamma^c$ and
\begin{align}
w(x^+)-v(x^-) &= 0,\\
\frac{\partial w(x^+)}{\partial n}-\frac{\partial v(x^-)}{\partial n} &= 0,
\end{align}
where $w(x^+)$ and $v(x^-)$ denote the limiting function values as $x$ approaches 
$\Gamma_2$ from $R_-^2\backslash\Omega_2$ and $\Omega_2$, respectively.

As in the first part of the proof, we now choose a sufficiently large half disk, 
still denoted by $D$, in $R_-^2\backslash\Omega_2$ that contains $\Gamma_2$. 
Applying the first Green's theorem to $D\backslash \Omega_2$ and $\Omega_2$, eq. \eqref{eq18},
and the various jump conditions above lead to: 
\begin{equation}
\text{Im}\bigg(\int_{\partial \Omega_2}w\frac{\partial \overline{w}}{\partial n}\bigg) = 0 \, .
\end{equation}
Once again, from Rellich's theorem and unique continuation we may conclude that $v=0$ in $\Omega_2$, 
which implies that (a) $\mu = 0$ on $B\cup \Gamma$ from \eqref{eq18}, and 
(b) $\mu = 0$, $\sigma = 0$ on $\Gamma_1$ from \eqref{eq12}. This completes the proof.
\end{proof}

\begin{remark}
The cavity problem can also be solved through a variational formulation. 
The corresponding well-posedness was studied in \cite{HGA1,HGA3}.
\end{remark}

\section{Numerical discretization}
\label{section_numerical_discretization}
For our numerical simulations, we will make use of Nystr\"om discretization of the 
integral equation \eqref{eq11}. Since the kernels in the integral equation are logarithmically
singular, this requires some care. Fortunately, for smooth boundaries, there are by now a host
of simple, high-order rules available (see, for example,~\cite{Alpert1999,BGR2010,Helsing2008,klockner2013quadrature}). 
Here, we use composite Gaussian quadrature. Following the discussion of \cite{GHL2014}, we 
divide each smooth component of the boundary into $N$ curved panels with $p$ points in each panel. 
The $p$ points are chosen as scaled Gaussian-Legendre nodes, so that for smooth integrands
the order of accuracy is $(2p-1)$. More precisely, we replace the integral
\begin{equation}
\int_{\Gamma} K(\mathbf{x},\mathbf{y}) \sigma(\mathbf{y}) dS_{\mathbf{y}}
\end{equation} 
by the quadrature
\begin{equation}
\sum_{i = 1}^{N} \sum_{j=1}^{p} \mathcal{K}(\mathbf{x}_{l,m},\mathbf{y}_{i,j}) \sigma(\mathbf{y}_{i,j}) w_{l,m,i,j}
\end{equation}
where $\mathbf{x}_{l,m}$ is the $m$-th Gauss-Legendre node on panel $l$, $\mathbf{y}_{i,j}$ is the $j$-th Gauss-Legendre node on panel $i$, $w_{l,m,i,j}$ is the quadrature weight and $\mathcal{K}$ is the "quadrature kernel". For non-adjacent panels, we simply set
$\mathcal{K}(\mathbf{x}_{l,m},\mathbf{y}_{i,j})= K(\mathbf{x}_{l,m},\mathbf{y}_{i,j})$,
for node $\mathbf{x}_{l,m}$ on panel $l$ and node $\mathbf{y}_{i,j}$ on panel $j$. 
The weights $w_{l,m,i,j}$ in that case are the standard Gauss-Legendre weights at 
$\mathbf{y}_{i,j}$ scaled to the length of the $i$-th panel. 
For the self-interaction of a panel ($i=l$), or the interaction with an adjacent panel,
the quadrature kernel is computed by generalized Gaussian quadrature~\cite{BGR2010}. 
From a linear algebra point of view, the self and adjacent panel interactions correspond to the block 
tridiagonal entries in the matrix. In our direct solver, these elements are precomputed and stored. 
All other matrix entries are computed on the fly.

In the presence of corners, 
a graded mesh~\cite{Brem2012,Helsing2008} is used to maintain high order accuracy. 
More specifically, after uniform discretization by panels, we perform a dyadic refinement for the 
panels that impinges on each corner point. A $p$-th order generalized Gaussian quadrature is 
used on each of the refined panels. The formal error analysis for such discretization is 
rather involved, since it depends on the regularity of $\sigma$ and $\mu$. 
Readers are referred to~\cite{Brem2012,CHAN1983,Helsing2008} and the references therein for 
a discussion of the relevant analysis. Typically, 
if $\epsilon$ denotes the length of the finest panel in the refinement, then the error is 
proportional to $\mathcal{O}(e^{-p}+\epsilon)$. Since the mesh is dyadically refined, this requires
$\mathcal{O}(\log(1/\epsilon))$ additional nodes at each corner. 

With a total of $N$ nodes, solving
the linear system corresponding to the above discretization by conventional Gaussian
elimination requires $\mathcal{O}(N^3)$ work - an arduous task if the cavity has a complicated boundary
or if the wavenumber of the incoming field is high. If the wave number $k=1000$, for example,
and if the length of the boundary is $20$ (in normalized units), then 
the number of unknowns is around $100,000$, assuming $20$ points per wavelength and counting 
each complex unknown as two real unknowns. Conventional solvers would take a few days to 
factor this system on a single core machine operating at $2.6$GHz. 
The fast direct solver described in the next section takes around $10-15$ minutes.

\section{Fast direct solver}
\label{section_fast_direct_solver}
In this section, we briefly outline the fast direct solver we will use to solve the discretized
integral equation~\eqref{eq11}. 
It is a simple extension of the method described in~\cite{ambikasaran2013fast}, to which
we refer the reader for a more complete description of the method.
As noted in the introduction, there are two strong arguments
in favor of this strategy: 
\begin{itemize}
\item
The performance of the direct solver is insensitive to multiple reflections inside the cavity.
At high frequencies, this causes the problem to be 
{\em physically} ill-conditioned and causes severe degradation in the convergence of iterative
methods.
\item
The solver is particularly effective for multiple right hand sides. It proceeds in two 
steps: first, the  construction of a fast hierarchical factorization, and second the application
of the factored inverse to each new right-hand side, at much lower cost.
\end{itemize}

\subsection{Hierarchical off-diagonal low-rank matrices}
Discretizing the integral equation~\eqref{eq8} along the curve (a one-dimensional manifold)
and ordering the unknowns and equations sequentially, yields a linear system of the form $A\sigma = b$, 
where $A \in \mathbb{R}^{N \times N}$, $\sigma,b \in \mathbb{R}^{N}$. 
Because potential-theoretic interactions are smooth in the far field, the 
matrix $A$ has a hierarchical off-diagonal low-rank (HODLR) structure, 
as observed in~\cite{ambikasaran2013fast, amirhossein2014fast, kong2011adaptive}.
(Related formalisms that also permit fast solution can be found in 
\cite{chandrasekaran2006fast1,chandrasekaran2006fast,ho2012fast,martinsson2009fast,martinsson2005fast},
\cite{bebendorf2005hierarchical,borm2003hierarchical,borm2003introduction,hackbusch2001introduction,ambikasaran2013thesis}, 
and the references therein.)

For illustration, we note that
a $2$-level HODLR matrix can be written in the form shown in 
equation~(\ref{equation_HODLR_2_level}).
\begin{align}
A & =
\begin{bmatrix}
A_1^{(1)} & U_{1}^{(1)} K_{1,2}^{(1)}V_{2}^{(1)^T}\\
U_{2}^{(1)} K_{2,1}^{(1)} V_{1}^{(1)^T} & A_2^{(1)}
\end{bmatrix}\\
& =
\begin{bmatrix}
\begin{bmatrix}
A_{1}^{(2)} & U_{1}^{(2)} K_{1,2}^{(2)} V_{2}^{(2)^T}\\
U_{2}^{(2)} K_{2,1}^{(2)} V_{1}^{(2)^T} & A_{2}^{(2)}
\end{bmatrix}
&
U_{1}^{(1)} K_{1,2}^{(1)} V_{2}^{(1)^T}\\
U_{2}^{(1)} K_{2,1}^{(1)} V_{1}^{(1)^T}&
\begin{bmatrix}
A_{3}^{(2)} & U_{3}^{(2)} K_{3,4}^{(2)} V_{4}^{(2)^T}\\
U_{4}^{(2)} K_{4,3}^{(2)} V_{3}^{(2)^T} & A_{4}^{(2)}
\end{bmatrix}
\end{bmatrix}
\label{equation_HODLR_2_level}
\end{align}
Each off-diagonal block is of low-rank (i.e., the rank of these blocks does not grow with the 
system size), although these ranks can all be different.
In general, for a $\kappa$-level HODLR matrix $A$, the $i^{th}$ diagonal block at level $k$, where $1 \leq i \leq 2^{k}$ and $0 \leq k < \kappa$, denoted by $A_{i}^{(k)}$, can be written as
\begin{align}
A_{i}^{(k)} =
\begin{bmatrix}
A_{2i-1}^{(k+1)} & U_{2i-1}^{(k+1)} K_{2i-1,2i}^{(k+1)}V_{2i}^{(k+1)^T}\\
U_{2i}^{(k+1)} K_{2i, 2i-1}^{(k+1)} V_{2i-1}^{(k+1)^T} & A_{2i}^{(k+1)}
\end{bmatrix} \, ,
\label{equation_HODLR_level_k}
\end{align}
where $A_{i}^{(k)} \in \mathbb{R}^{N/2^k \times N /2^k}$, and 
$U_{2i-1}^{(k)}, U_{2i}^{(k)}, V_{2i-1}^{(k)}, V_{2i}^{(k)}$ are thin matrices with $\dfrac{N}{2^k}$ 
rows. A pictorial representation of the matrix $A$ is shown in Figure~\ref{figure_HODLR_matrix}.
\begin{figure}[!htbp]
\includegraphics[scale=1]{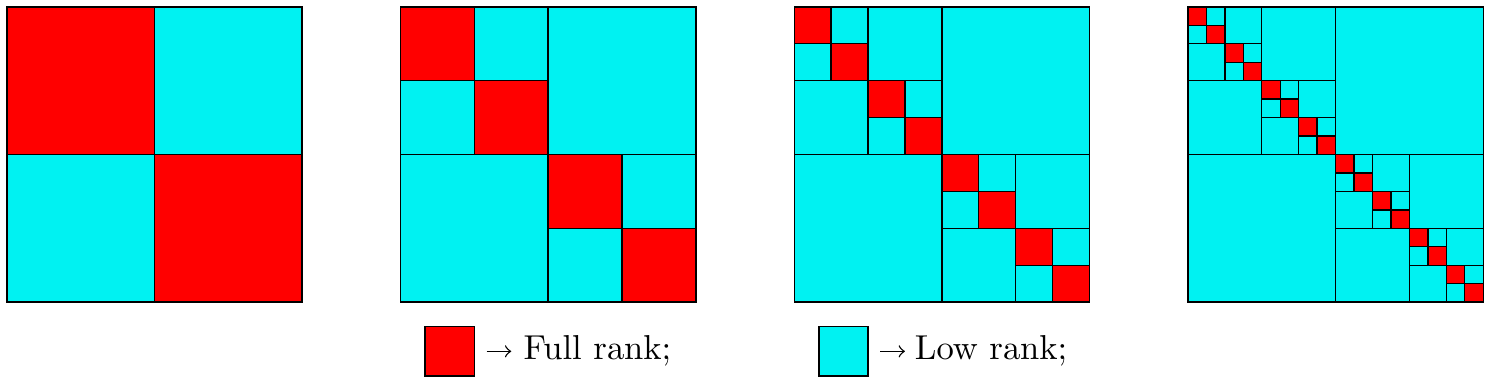}
\caption{A hierarchical off-diagonal low-rank matrix at different levels.}
\label{figure_HODLR_matrix}
\end{figure}

For the solver developed here, the low-rank decomposition of the off-diagonal blocks is obtained 
using the adaptive cross approximation (ACA)~\cite{rjasanow2002adaptive, zhao2005adaptive} algorithm, 
which is a modification of the partially pivoted LU algorithm. The advantage of ACA is that the 
computational cost of obtaining a low-rank factorization of a low-rank $M \times N$ matrix is 
$\mathcal{O}_{\epsilon}(M+N)$. The technique is 
based entirely on numerical linear algebra, so the low-rank construction of the off-diagonal blocks is 
independent of the underlying integral operator. Once the low-rank decomposition of 
the off-diagonal blocks at all levels is obtained, the factorization of the matrix $A$ 
proceeds along the lines described in~\cite{ambikasaran2013fast}. In other words, 
the matrix $A$ is factored as shown in eq. ~\eqref{equation_HODLR_factorization}.
\begin{align}
A = A_{\kappa} A_{\kappa-1} \cdots A_1 A_0 \, ,
\label{equation_HODLR_factorization}
\end{align}
where the $A_i$'s are block diagonal matrices with $2^i$ diagonal blocks and each block is a 
low-rank perturbation of the identity matrix. The factorization can be obtained at a computational 
cost of the order $\mathcal{O}(N \log^2 N)$ by recursive application of the 
Sherman-Morrison-Woodbury formula. 
A pictorial representation of the above factorization for a level $3$ HODLR matrix is shown in 
Figure~\ref{figure_HODLR_factorization}.
\begin{figure}[!htbp]
\includegraphics[scale=0.925]{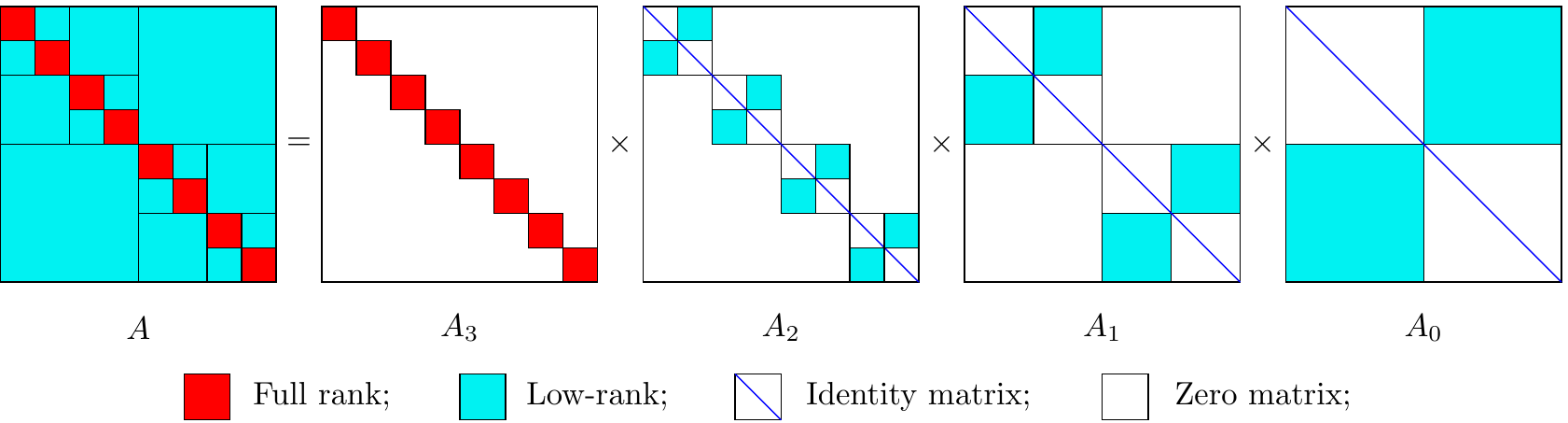}
\caption{Factorization of a HODLR matrix at level $3$.}
\label{figure_HODLR_factorization}
\end{figure}
Other fast methods, such as the interpolative decomposition~\cite{liberty2007randomized} can be used
in place of ACA. Further details can be found in~\cite{ambikasaran2013fast} and an
implementation is available from~\cite{HODLRSolver}.
 
As with any direct factorization, a principal advantage is the performance of the scheme
with multiple right hand sides. That is, the cost of solving the linear system with `$r$' 
right-hand sides scales as $\mathcal{O}(r N \log N)$, with a smaller constant than for the factorization 
step. This is very attractive in the present context, since it permits the
computation of the scattered field for multiple incident angles in a negligible amount of time.
Each problem corresponds to a new right-hand side in the integral equation.


\section{Numerical results}
\label{section_numerical_results}
In this section, we illustrate the performance of our algorithm on three different cavity shapes. 
In general, there is no exact solution for the field scattering from a cavity embedded in a
ground plane. We may, however, validate our solver by using ``artificial" boundary data generated 
by a collection of point sources, in which case the potential field is known and can be used 
for comparison with the potential generated by the integral representation on the boundary. 
More specifically, we define the following field in the perturbed upper half-space
$R_+^2\cup \Omega_1$:
\begin{equation}
u(\mathbf{x}) = \frac{i}{4}H^{(1)}_0(k|(\mathbf{x}-\mathbf{x_0})|)+\frac{i}{4}H^{(1)}_0(k|(\mathbf{x}-\mathbf{x'_0})|).
\end{equation}
We then solve the following system of equations
\begin{equation}
\left\{\begin{array}{rl}
-\mu + \mathcal{D}_{\Gamma}\mu &= u(\mathbf{x}), \quad \mbox{ for } \mathbf{x} \in \Gamma_1\\
\sigma + \mathcal{T}_{\Gamma}\mu &= \frac{\partial u(\mathbf{x})}{\partial n},\quad \mbox{ for } \mathbf{x} \in \Gamma_1\\
-\frac{1}{2}\mu+\mathcal{S}^{H}_{\Gamma_1}\sigma+ \mathcal{D}^{H}_{\Gamma_1} \mu+\mathcal{D}_{B}\mu + \mathcal{D}_{\Gamma}\mu &= u(\mathbf{x}), \quad \mbox{ for } \mathbf{x}\in B\cup \Gamma
\end{array}\right.
\end{equation}
where $\mathbf{x_0}=(5,12)$, $\mathbf{x'_0}=(5,-12)$ and the center of the cavity is at $(0.5,0)$.
It is straightforward to see that the 
field given by the representation \eqref{eq10} should be equal to $u(x)$ in $\Omega_1$ and zero in $R_+^2\backslash \Omega_1$.

\vspace{.2in}

\hrule
When presenting the numerical results, we use the following notation:
\begin{itemize}
\item $N_{mid}$: Number of equal-sized curve segments used to discretize each smooth component of the 
piecewise-smooth boundary.
\item $N_{corner}$: Number of panels used to dyadically refine the first and last panel on each smooth
component.  
\item $N_{tot}$: Total number of unknowns in the system
\item $T_{factor}$: Amount of time in second to factorize the system
\item $T_{solve}$: Amount of time in second required to solve the system after factorization
\item $E_{error}$: The average relative error at some random points in $\Omega_1$.
\end{itemize}
\hrule

\vspace{.2in}

In the examples below, when solving the true scattering problem, the incoming field is
assumed to be a plane wave and we compute the backscatter radar cross section (RCS) for each
cavity.  The backscatter RCS is defined as the intensity of the far field pattern in the same
direction as the incident angle. When the boundary of the cavity is strictly below the ground plane, 
it can be shown the backscatter RCS is given by~\cite{JIN3}:
\begin{equation}
RCS = \frac{4}{k}\bigg|\frac{k}{2}\sin\theta \int_{\Gamma} e^{ik\cos\theta x}dx\bigg|^2,
\end{equation} 
where $\theta$ is the angle of incidence and $\Gamma$ is the aperture of the cavity. 
If part of the boundary is above the ground plane, the RCS can be found from $\eqref{eq10}$ 
and the asymptotic behavior of the Hankel function.

We choose the artificial boundary to be a half circle centered at $(0.5,0)$, 
the center of the aperture, with a radius of $2.5$. 
For numerical stability, we use inner-product-preserving Nystr\"om scaling, as
proposed in~\cite{Brem2012}. That is, the discrete unknowns are taken to be the physical unknowns
multiplied by the square root of the corresponding quadrature weight. 
All numerical tests have been carried out on a laptop with $4$Gb memory and $2.6$GHz Intel CPU.

\subsection{Example $1$: Pot shaped cavity}
In our first example, we consider the pot shaped cavity shown in Figure~\ref{fig11_pot_shaped_cavity}. 
The width of the aperture is $1$, below which is a circle of radius $\sqrt{2}/2$ 
centered at $(0.5,-1)$. 
The accuracy of the solution for wavenumbers ranging from $k=1$ to $k=800$ are 
shown in Table \ref{table_1}.
 
With an order of accuracy $p=10$,
we discretize the boundary with roughly $50$ points per wavelength to achieve an accuracy of $10^{-9}$. 
Table~\ref{table_1} shows that the required time scales approximately as 
$\mathcal{O}(k N_{tot})$.
Note that accuracy is lost more or less linearly with increasing wavenumber.
The results support our observation above that the fast direct solver is 
particularly efficient at computing the RCS for
multiple angles of incidence, since the cost for each new right-hand side is very modest.
In Figure~\ref{fig11_pot_shaped_cavity}, we plot the scattered field for a normally incident plane wave 
at $k=160$, as well as the backscatter RCS for angles ranging from $0$ to $\pi$, sampled at
$360$ equispaced steps.
The time taken for the factorization is less than a minute and the total time taken to solve for 
all incident angles takes around $30$ seconds. 
\begin{table*}[!htbp]
\begin{center}
\caption{Results for the pot shaped cavity over a range of wavenumbers}
\label{table_1}
\begin{tabular}{|c|c|c|c|c|c|c|}
\toprule
Wavenumber $k$ & $N_{corner}$ & $N_{mid}$ & $N_{tot}$ & $T_{factor}$ & $T_{solve}$ & $E_{error}$ 
\\
\midrule
$1$ & 10  & 2 & $1320$  & $0.5$ & 0.01 & $7.6 \cdot 10^{-10}$ \\
$10$ & 10 & 10 & $2200$ & $0.65$ & 0.01 & $5 \cdot 10^{-8}$ \\
$20$ & 10 & 20 & $3300$ & $1.52$ & 0.01 &$1.2 \cdot 10^{-9}$ \\
$40$ & 10 & 40 & $5500$ & $3.14$ & 0.02 & $8.5 \cdot 10^{-9}$ \\
$80$ & 10 & 80 & $9900$ & $9.5$ & 0.05 & $4.1 \cdot 10^{-9}$ \\
$160$ & 10 & 160 & $18700$ & $42.4$& 0.19 &$1.2 \cdot 10^{-9}$ \\
$320$ & 10 & 320 & $36300$ & $192.8$ & 0.57 & $2.1 \cdot 10^{-8}$ \\
$640$ & 10 & 400 & $45100$ & $581.4$ & 1.43 & $5.5 \cdot 10^{-5}$ \\
$800$ & 10 & 400 & $45100$ & $785.1$ & 1.57 & $4.6 \cdot 10^{-6}$ \\
\bottomrule
\end{tabular}
\end{center}
\end{table*}
 
\begin{figure}[!htbp]
\begin{center}
\subfloat[]{} \includegraphics[scale=0.4]{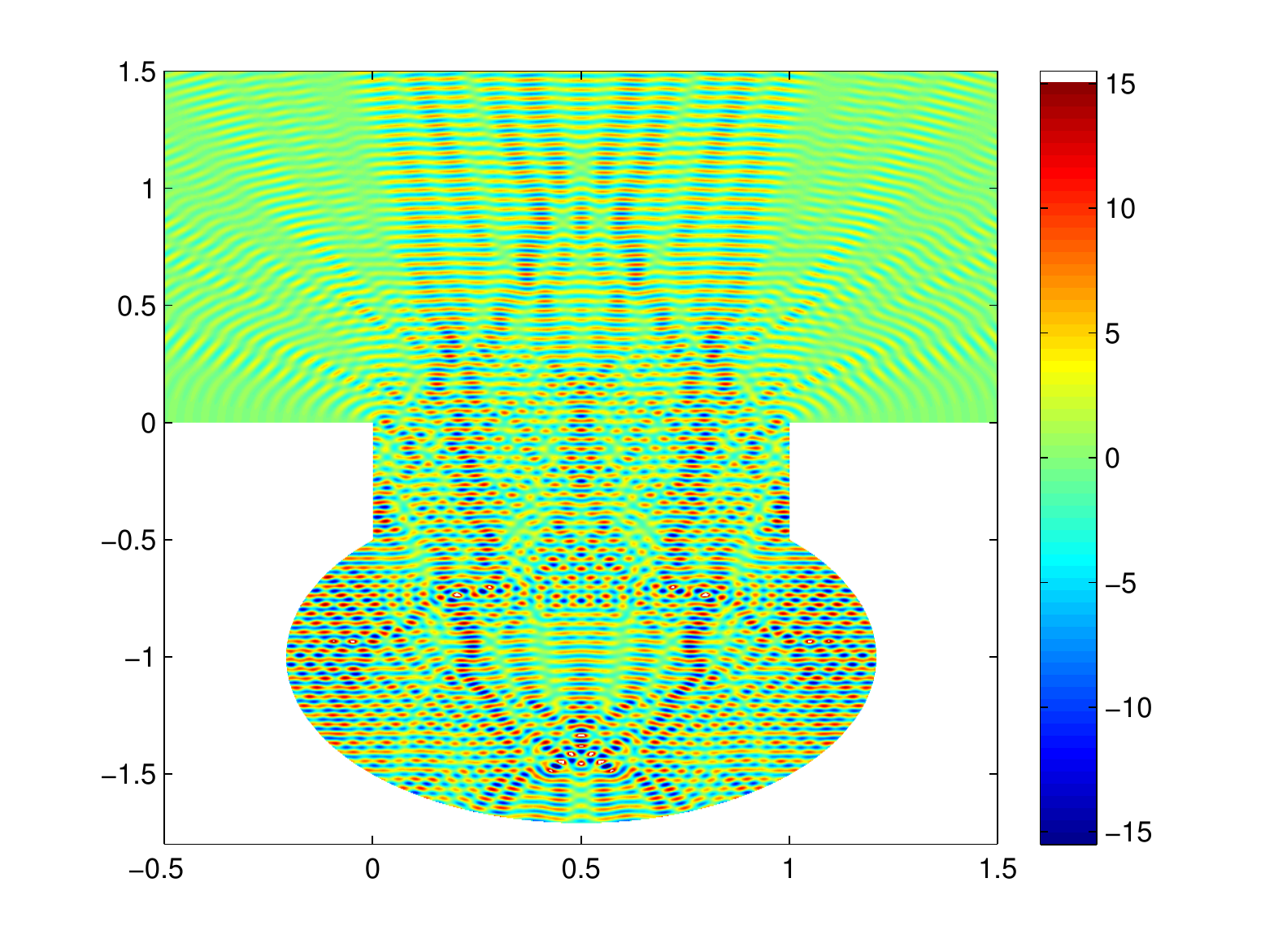}
\subfloat[]{} \includegraphics[scale=1]{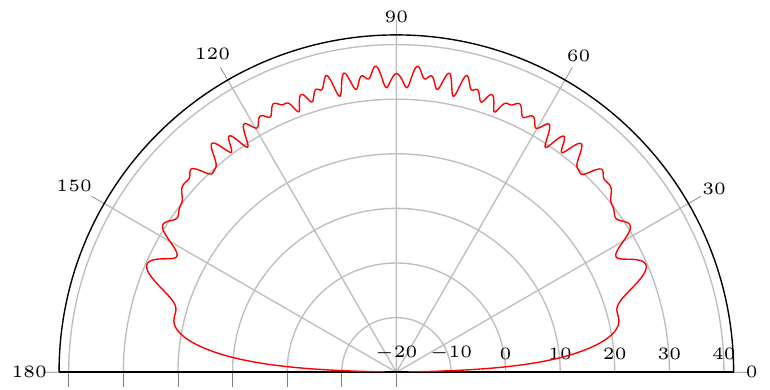}
\end{center}
\caption{{ \emph{Example 1. (a) Real part of the scattered field for a pot shaped cavity with
a normally incident plane wave at wavenumber k=160.  (b) The backscatter RCS in $dB$ for the pot shaped cavity at k=160}}}
\label{fig11_pot_shaped_cavity}
\end{figure}

\FloatBarrier

\subsection{Example 2: Engine shaped cavity} 

We next investigate scattering by an engine-shaped cavity. 
The cavity is constructed by connecting the points
$[0,0]$, $[0,-2]$, $[0.45,-2.0]$, $[0.45,-1.6]$, $[0.1,-1.6]$, $[0.1,-1.0]$, $[0.45,-1.0]$, 
$[0.45,-0.4]$, $[0.5,0.2]$, $[0.55,-0.4]$, $[0.55,-1.0]$, $[0.9,-1.0]$, $[0.9,-1.6]$, $[0.55,-1.6]$, 
$[0.55,-2.0]$, $[1,-2]$ and $[1,0]$, as shown in Figure \ref{fig11_engine_shaped_cavity}. 
Note that the tip of the engine is above the ground plane. 

Computational results are shown in Table~\ref{table_2}. Compared with Example 1, 
the required solution time is greater for the same wavenumber and approximately the same number of 
points. This is due to the fact that some segments of the boundary are physically close to each other, 
which increases the rank of some off-diagonal blocks in the HODLR matrix. 
The scattered field for a plane wave at angle of incidence $45^o$ with $k=160$ is plotted 
in Figure~\ref{fig11_engine_shaped_cavity}. 
The RCS is also shown in Figure~\ref{fig11_engine_shaped_cavity}. An animation, Movie 1, of the scattered field as the angle of incidence varies is provided in the supplementary material. Opening the animation using Adobe Reader XI, Version 11.0.06 is recommended.

\begin{table*}[!htbp]
\begin{center}
\caption{Results for engine-shaped cavity over a range of wavenumbers}
\label{table_2}
\begin{tabular}{|c|c|c|c|c|c|c|}
\toprule
Wavenumber $k$ & $N_{corner}$ & $N_{mid}$ & $N_{tot}$ & $T_{factor}$ & $T_{solve}$ & $E_{error}$  \\
\midrule
$1$ & 10  & 2 & $3360$  & $3.1$ & 0.01 & $8.8 \cdot 10^{-10}$\\
$10$ & 10 & 10 & $5600$ & $5.46$ & 0.03 & $7.7 \cdot 10^{-12}$\\
$20$ & 10 & 20 & $8400$ & $8.73$ & 0.04 &$5.0 \cdot 10^{-11}$  \\
$40$ & 10 & 40 & $14000$ & $17.83$ & 0.1 & $5.9 \cdot 10^{-9}$ \\
$80$ & 10 & 80 & $25200$ & $44.16$ & 0.26 & $7.8 \cdot 10^{-8}$ \\
$160$ & 10 & 160 & $47600$ & $155.2$& 0.58 &$4.2 \cdot 10^{-8}$ \\
$320$ & 10 & 200 & $58800$ & $466.3$ & 1.27 & $3.3 \cdot 10^{-8}$\\
$640$ & 10 & 200 & $58800$ & $980.4$ & 2.50 & $3.5 \cdot 10^{-8}$\\
$800$ & 10 & 200 & $58800$ & $1100.1$ & 3.01 & $9.6 \cdot 10^{-8}$\\
\bottomrule
\end{tabular}
\end{center}
\end{table*}

\begin{figure}[!htbp]
\begin{center}
\subfloat[]{} \includegraphics[scale=0.4]{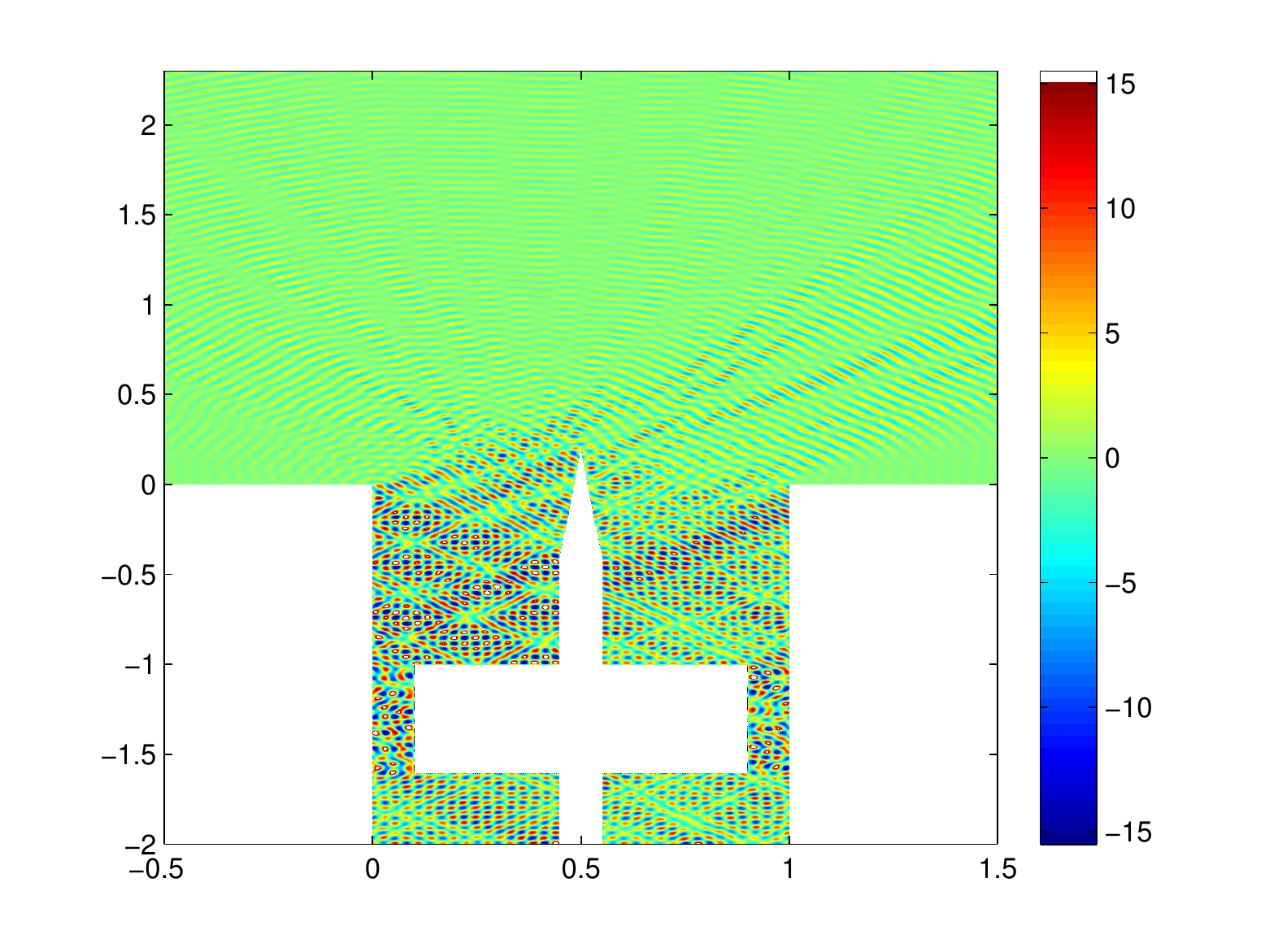}
\subfloat[]{} \includegraphics[scale=1]{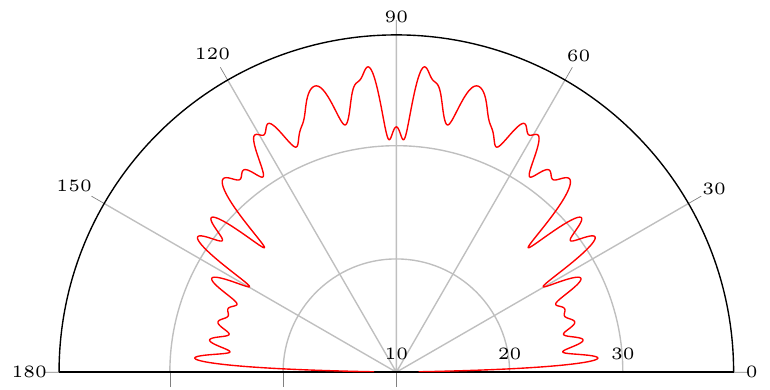}
\end{center}
\caption{{\emph{Example 2: (a) Real part of the scattered field for an engine-shaped cavity 
in response to an incident plane wave at angle $45^o$ with wavenumber $k=160$,  
(b) The backscatter RCS in $dB$ for the engine shaped cavity at $k=160$.}}}
\label{fig11_engine_shaped_cavity}
\end{figure}

\FloatBarrier

\subsection{Example 3: Rough bottom cavity}
In our last example, we consider a cavity with a rough-shaped boundary, which is extremely difficult 
to analyze by standard finite difference or finite element methods, especially at large wavenumber. 
The boundary of the cavity is parametrization by:

\begin{equation}
\left\{\begin{array}{l}
x(\theta) =r(\theta)\cos(\theta)+0.5 \\
y(\theta) =1.25r(\theta)\sin(\theta) 
\end{array}\right.
\end{equation}
where 
\begin{equation}
r(\theta) = 1+0.1\sin(2\theta)+0.1\sin(11\theta)+
0.08\sin(19\theta)\sin(29\theta)+0.05\sin(47\theta)
\end{equation}
and $\theta \in [\pi,2\pi]$.
Such rough structures may appear in modeling manufacturing defects or during intermediate steps
of an optimization procedure aimed at reconstructing an unknown cavity by solving the 
inverse scattering problem. The boundary is discretized using three segments. 
The results in Table~\ref{table_3} again show that the work is of the order
$\mathcal{O}(k N_{tot})$.
Figure~\ref{fig11_rough_bottom_cavity} shows the scattered field for $k=160$ with a plane
wave at normal incidence, as well as the RCS plot. 
Note that the rough bottom exhibits a much larger RCS between 
$\pi/4$ and $3\pi/4$ than in the previous examples. 

\begin{table*}[!htbp]
\begin{center}
\caption{Results for rough-bottom cavity over a range of wavenumbers}
\label{table_3}
\begin{tabular}{|c|c|c|c|c|c|c|}
\toprule
Wavenumber $k$ & $N_{corner}$ & $N_{mid}$ & $N_{tot}$ & $T_{factor}$ & $T_{solve}$ & $E_{error}$  \\
\midrule
$1$ & 10  & 20 & $3300$  & $2.3$ & 0.01 & $1.3 \cdot 10^{-8}$\\
$10$ & 10 & 20 & $3300$ & $2.1$ & 0.02 & $6.6 \cdot 10^{-9}$\\
$20$ & 10 & 40 & $5500$ & $3.2$ & 0.02 &$5.3 \cdot 10^{-8}$  \\
$40$ & 10 & 80 & $9900$ & $6.5$ & 0.05 & $4.8 \cdot 10^{-8}$ \\
$80$ & 10 & 160 & $18700$ & $20.8$ & 0.12 & $1.2 \cdot 10^{-8}$ \\
$160$ & 10 & 320 & $36300$ & $78.7$& 0.35 &$5.7 \cdot 10^{-8}$ \\
$320$ & 10 & 400 & $45100$ & $247.3$ & 0.73 & $2.9 \cdot 10^{-8}$\\
$640$ & 10 & 400 & $45100$ & $772.4$ & 1.38 & $4.2 \cdot 10^{-7}$\\
\bottomrule
\end{tabular}
\end{center}
\end{table*}

\begin{figure}[!htbp]
\begin{center}
\subfloat[]{} \includegraphics[scale=0.4]{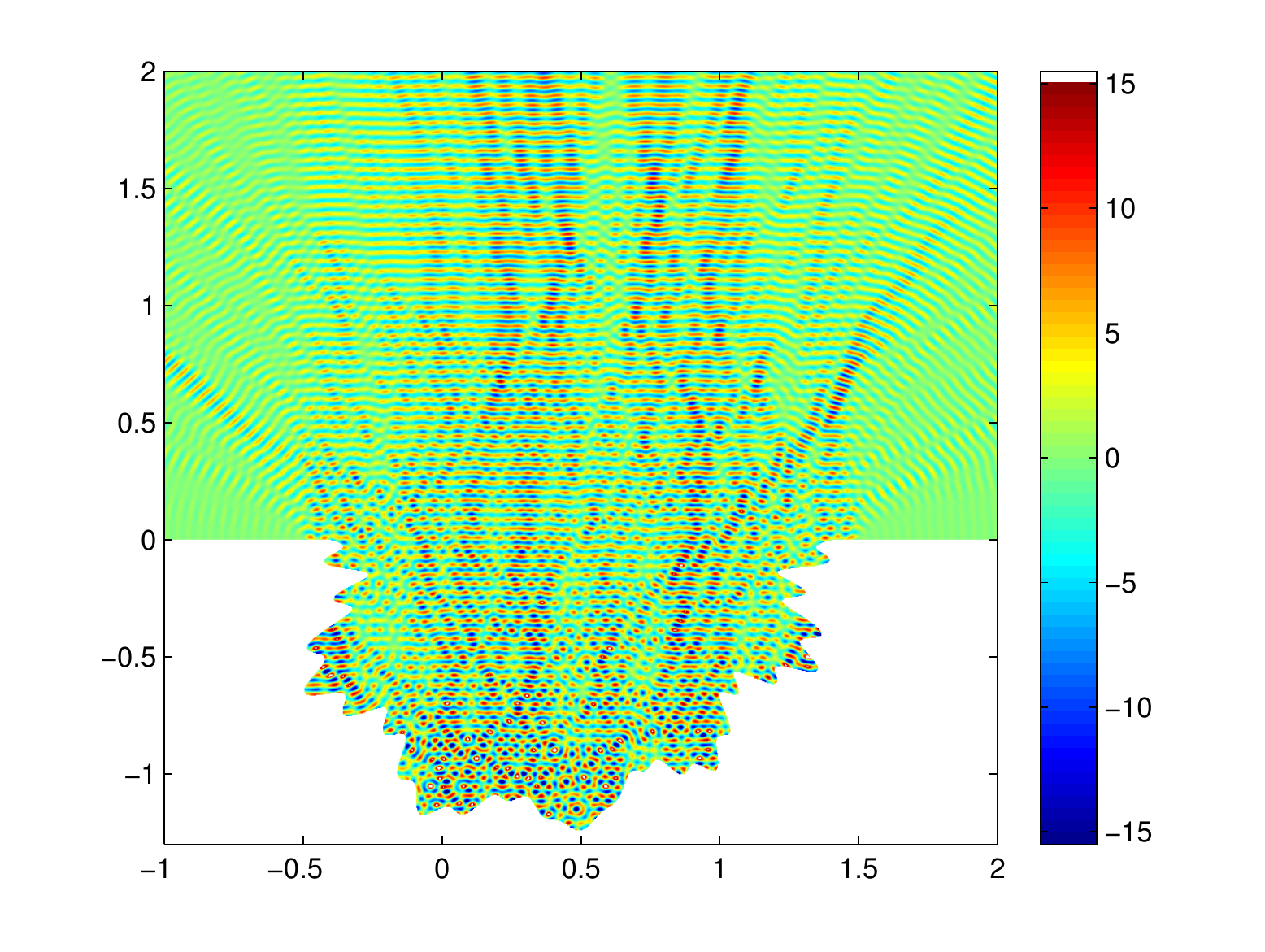}
\subfloat[]{} \includegraphics[scale=1]{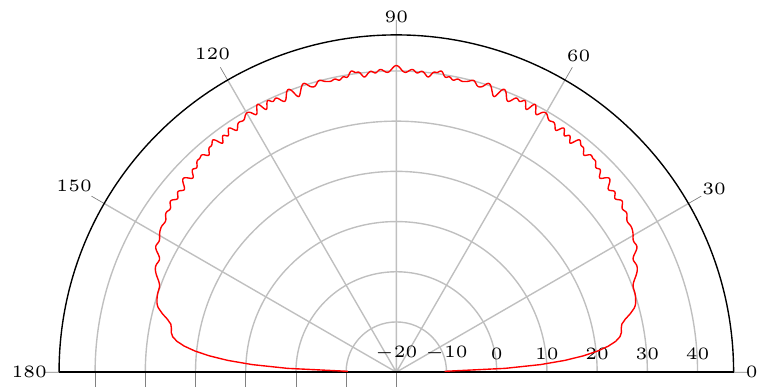}
\end{center}
\caption{{ \emph{Example 3. (a) Real part of the scattered field for a rough bottom cavity with a 
normally incident plane wave at wavenumber k=160.  (b) The backscatter RCS in $dB$ for the rough bottom cavity at k=160.}}}
\label{fig11_rough_bottom_cavity}
\end{figure}

\FloatBarrier


\section{Conclusion}
\label{section_conclusion}
We have presented a new integral formulation for high frequency electromagnetic scattering from a 
large cavity that leads to a second kind integral equation and which is compatible with 
a fast and accurate direct solver. The main novelty is the use of a global density, 
which is non-physical in the sense that the field in certain domains 
is determined by layer potential densities that are not necessarily on the 
boundary of the domain itself. We have proven well-posedness of the formulation, which implies that
the method does not suffer from any spurious resonances. The equation is discretized by 
high order quadrature and numerical experiments show that the direct solver is very effective even 
for large frequencies and arbitrarily-shaped cavities. Future work includes extending the 
formulation to an impedance boundary, to the full Maxwell equations in three dimensions, and 
to problems of optimal design.

\section{Acknowledgements}
This work was supported by the Applied Mathematical Sciences Program of the U.S. Department of Energy under Contract DEFGO288ER25053 and by the Office of the Assistant Secretary of Defense for Research and Engineering and AFOSR under NSSEFF Program Award FA9550-10-1-0180.

\bibliographystyle{abbrv}

\end{document}